\documentclass[12pt]{amsart}

\usepackage{a4wide}
\usepackage{amsmath,amssymb} 
\usepackage{bbm}
\usepackage{color}
\usepackage{tikz}
\usepackage{tkz-graph}
\tikzset{
  LabelStyle/.style = { rectangle, rounded corners, draw,
                        minimum width = 1em, 
                        font =  },
  VertexStyle/.append style = { inner sep=3pt,
                                font = \large},
  EdgeStyle/.append style = {->, bend left} }

\tikzset{
    ultra thick/.style={line width=5.0pt}
}

\theoremstyle{plain}
\newtheorem{theorem}{Theorem}
\newtheorem{prop}[theorem]{Proposition}
\newtheorem{lemma}[theorem]{Lemma}
\newtheorem{coro}[theorem]{Corollary}

\theoremstyle{definition}
\newtheorem{definition}[theorem]{Definition}
\newtheorem{remark}[theorem]{Remark}

\newtheorem{example}[theorem]{Example}

\newcommand{\Z}{{\mathbb Z}}
\newcommand{\Q}{{\mathbb Q}}

\newcommand{\N}{{\mathbb N}}

\newcommand{\mc}{\mathcal}

\newcommand{\freq}{\operatorname{freq}}

\newcommand{\Per}{\operatorname{Per}}
\newcommand{\Orb}{\operatorname{Orb}}
\newcommand{\per}{\operatorname{per}}
\newcommand{\dsub}{\varphi}

\begin{document}

\title[Periodic points in RS-subshifts]{
Periodic points in random substitution subshifts
}

\author[D.~Rust]{Dan Rust}
\date{\today}
\address{Fakult\"at f\"ur Mathematik, Universit\"at Bielefeld, \newline
\hspace*{\parindent}Postfach 100131, 33501 Bielefeld, Germany}
\email{drust@math.uni-bielefeld.de}

\begin{abstract}
We study various aspects of periodic points for random substitution subshifts.
In order to do so, we introduce a new property for random substitutions called the disjoint images condition.
We provide a procedure for determining the property for compatible random substitutions---random substitutions for which a well-defined abelianisation exists.
We find some simple necessary criteria for primitive, compatible random substitutions to admit periodic points in their subshifts.
In the case that the random substitution further has disjoint images and is of constant length, we provide a stronger criterion.
A method is outlined for enumerating periodic points of any specified length in a random substitution subshift. 

\end{abstract}



\maketitle

Random substitutions are a generalisation of the classical notion of a substitution on a finite alphabet.
In the classical setting, letters are mapped to words over the same alphabet, and then this map is iterated.
Dynamical systems associated with these classical substitutions are well studied and there is a large community devoted to solving some of the few remaining big problems in this area \cite{bible,DL:classification,F:book}.
In the setting of random substitutions, letters have a \emph{set of possible words} (often with an accompanying probability distribution) to which they may be independently mapped.
There, one must contend with all possible outcomes of iteration, where each letter of a word is mapped independently of all others.
This leads to an exponential growth in the number of words admitted in the language of a random substitution and a corresponding explosion of complexity for the associated subshift of bi-infinite sequences over that language.
Accordingly, the dynamical systems and tilings associated with random substitutions provide good models for quasicrystaline structures that have long range order induced by an underlying hierarchical supertile structure, whilst also possessing positive entropy.
Such models are highly sought-after in the world of solid-state physics \cite{GL:random} and have also proved useful for the study of molecular evolution \cite{K:thesis,MG:random-dna} where so-called expansion-modification systems are a model proposed to explain long-range correlations of sequences associated with DNA.

The recent study of random substitutions has lead to rapid advances in our understanding of various topological, dynamical and diffractive properties of their associated tilings and subshifts \cite{BM:noble-means,BSS:random,GRS:random-sft,MTU:random,RS:random}.
With this greater understanding has come a multitude of simple to state but non-trivial open problems, some of which were outlined in recent articles \cite{GRS:random-sft,RS:random}.
The purpose of this article is to tackle one particular aspect of the open problems presented there---namely the existence and enumeration of shift-periodic points.

Two of the simplest invariants for a topological dynamical systems are its topological entropy and its set of periodic points.
Any study of a newly defined class of dynamical systems should start with an attempt at understanding these features.
For subshifts of finite type for instance, both the entropy and periodic points are completely understood \cite{LM:introduction-to-symbolic} and prove to be useful for their further study and classification.
The entropy of random substitution subshifts has received recent attention \cite{BSS:random,M:diffraction-noble,N:fibonacci-entropy,N:random-one} where some examples have been explicitly calculated and some general results on entropy have also appeared \cite{GRS:random-sft,RS:random}.
This direction recently culminated in Gohlke's establishment of a general method for calculating topological entropy \cite{G:entropy} for compatible substitutions, where previous calculations relied on ad hoc methods.
His results allow for explicit converging bounds of entropy to be given and exact values in the case of certain large families.

The study of periodic points for random substitutions is comparatively unexplored.
For deterministic substitutions, Moss\'{e}'s celebrated result \cite{M:aperiodic} relating periodicity of the subshift with the notion of \emph{recognisability} highlights the importance of determining when a substitution admits periodic points in its subshift.
Determining when a primitive deterministic substitution has periodic points is relatively simple \cite{BR:grout} and the minimality of the subshift allows one to easily count the total number of periodic points.
In the case of random substitutions, the question of identifying the existence of periodic points is decidedly more difficult.
This is due to the intricate interplay between the shift dynamics and the long-range hierarchical structure induced by the inflation action of the random substitution, as well as the highly non-minimal nature of the subshift.
Our goal in this article is to understand this interplay and to lay the groundwork for the further study of periodic points.

In Section \ref{SEC:rs-subshift}, we introduce the basic notions of random substitutions and random substitution subshifts, here called \emph{RS-subshifts}. We discuss what it means to be compatible and primitive and we recall some basic properties of RS-subshifts from previous work \cite{RS:random}.

In Section \ref{SEC:disjoint-images}, we introduce a new property for random substitutions called the \emph{disjoint images} property.
The disjoint images property plays a key role in allowing natural preimage arguments to work for several proofs in Sections \ref{SEC:existence-periodic} and \ref{SEC:enumerating-periodic}.
We outline a general method for determining if and when a particular random substitution has disjoint images and we apply this method to examples.
We also discuss a weaker form of the property which is closely related to unique global recognisability.

In Section \ref{SEC:existence-periodic}, we study sufficient and necessary conditions for the existence of periodic points in RS-subshifts.
In particular we present some basic arguments, based on letter frequencies, in order to show that a primitive RS-subshift with periodic points must be associated with a random substitution whose expansion factor is an integer.
As long as we have disjoint images, we also show that periodic points in RS-subshifts have periodic substitutive preimages with least period no greater than the original point---this result is key to our concrete enumerations of periodic points.
The main result of the section is a simple sufficient criterion for the absence of periodic points in RS-subshifts associated with substitutions of constant length with disjoint images.

In Section \ref{SEC:enumerating-periodic}, we set the task of attempting to enumerate periodic points when they exist.
We provide a decidable procedure for determining if a particular word is a periodic block for some element in the RS-subshift of a compatible random substitution with disjoint images.
With a computer-assisted exact enumeration, we iteratively apply this procedure to words in the language of the \emph{random period doubling} substitution in order to enumerate all possible periodic blocks of length at most 30.

\section{Random substitution subshifts}\label{SEC:rs-subshift}
Let $\mc A = \{a_1, \ldots, a_d\}$ be a finite \emph{alphabet} whose elements are referred to as \emph{letters}.
Let $\mc A^n$ denote the set of \emph{words of length $n$} in $\mc A$ given by all concatenations of letters from $\mc A$ and for $u \in \mc A^n$, write $|u| = n$ for the length of the word $u$.
Let $\mc A^\ast = \bigcup_{n=0}^\infty \mc A^n$ denote the set of finite words in $\mc A$ with empty word $\varepsilon$ and let $\mc A^+ = \mc A^\ast \setminus {\varepsilon}$.
A word $u = u_0 \cdots u_k$ is a \emph{subword} of the word $v = v_0 \cdots v_{\ell}$ and we write $u \triangleleft v$ if there exists $j$ such that $u_i = v_{i+j}$ for $0 \leq i \leq k$.
We write $|v|_u = \#\{j \mid u_i = v_{i+j}, 0 \leq i \leq k\}$ to denote the number of occurrences of the word $u$ as a subword of the word $v$.
The \emph{cyclic permutation function} $\alpha \colon \mc A^* \to \mc A^*$ is given by $\alpha(u_1 u_2 \cdots u_k) = u_2 \cdots u_k u_1$.
Let $\mc A^\Z$ denote the set of bi-infinite sequences over the alphabet.
The set $\mc A^\Z$ forms a compact metrisable space under the product topology and the \emph{shift map} $\sigma \colon \mc A^\Z \to \mc A^\Z$ given by $\sigma(x)_n = x_{n+1}$ is a homeomorphism.

A function $\dsub \colon \mc A \to \mc A^+$ is called a \emph{deterministic substitution} and uniquely extends by concatenation to a morphism $\dsub \colon \mc A^\ast \to \mc A^\ast$.
Deterministic substitutions are well-studied \cite{bible, F:book}.
In contrast, a random substitution can take multiple values on a single letter $a \in \mc A$.
\begin{definition}
Let $\mc A $ be a finite alphabet, and let $\mc P (\mc A^+)$ denote the power set of $\mc A^+$.
A \emph{random substitution} on $\mc A$ is a map $\vartheta \colon \mc A \to \mc P(\mc A^{+})\setminus \varnothing$.
We say that $\vartheta$ has \emph{finite range} if $\# \vartheta (a)$ is a finite set of words for all $a \in \mc A$.
We call a word $u$ a \emph{realisation} of $\vartheta$ on $a$ if $u \in \vartheta(a)$.
If $u \triangleleft v$, with $v \in \vartheta(a)$, then we write $u \blacktriangleleft \vartheta(a)$.
\end{definition}
We will only be concerned with random substitutions with finite range and so we implicitly assume that all random substitutions from now on have finite range.
A word which can appear as a realisation $u \in \vartheta(a)$ for some $a \in \mc A$ is also called an \emph{inflation word}.

We can extend $\vartheta$ to a function $\mc A^{+} \to \mc P(\mc A^+)\setminus \varnothing$ by concatenation,
\[
\vartheta(a_1 \cdots a_m) = \vartheta(a_1) \cdots \vartheta(a_m) := \{ u_1 \cdots u_m \mid u_i \in \vartheta(a_i), \; 1 \leq i \leq m \},
\]
and consequently to a function $\vartheta \colon \mc P (\mc A^{+}) \setminus \varnothing \to \mc P (\mc A^+) \setminus \varnothing$ by $\vartheta(B) := \bigcup_{u \in B} \vartheta(u)$.
This then allows us to take powers of $\vartheta$ by composition, giving $\vartheta^k \colon \mc P (\mc A^{+}) \setminus \varnothing \to \mc P (\mc A^{+}) \setminus \varnothing $, for any $k \geq 0$, where $\vartheta^0 := \operatorname{id}_{\mc P(\mc A^+) \setminus \varnothing}$ is the identity and $\vartheta^{k+1} := \vartheta \circ \vartheta^k$.
We say that a word $u \in \mc A^{\ast}$ is \emph{$\vartheta$-legal} or just \emph{legal} if there is $k \in \N$ and $a \in \mc A$ such that $u \blacktriangleleft \vartheta^k(a)$.
We say that $\vartheta$ \emph{has constant length $\ell$} if there exists a natural number $\ell \geq 1$ such that $u \in \vartheta(a) \implies |u| = \ell$ for all $a \in \mc A$.
\begin{remark}
We should mention that the set $\vartheta(a)$ of realisations of the substitution on $a$ can be equipped with a probability distribution \cite{BSS:random,GL:random,RS:random}.
We have chosen to suppress the probabilities in this article for the sake of clarity.
The probabilities only affect measure theoretic properties and (under mild conditions) frequencies of words of length at least two.
Here, we are only interested in topological dynamics and frequencies of single letters respectively.
\end{remark}

\begin{example}\label{EX:fib}
The \emph{random Fibonacci substitution}, first studied by Godr\'{e}che and Luck \cite{GL:random}, is defined on the alphabet $\mc A = \{a,b\}$ by
\[
\vartheta \colon a \mapsto \{ab,ba\}, \: b \mapsto \{a\}.
\]
The next two iterates of the random Fibonacci substitution are then given by

\[
\begin{array}{rl}
\vartheta^2 \colon & a \mapsto \{aba,baa,aab\}, \:\: b \mapsto \{ab,ba\},\\
\vartheta^3 \colon & a \mapsto \{abaab,ababa,baaab,baaba,aabab,aabba,abbaa,babaa\},\\
                   & b \mapsto \{aba,baa,aab\}.
\end{array}
\]
So a realisation of $\vartheta^3(a)$ is given by $baaab \in \vartheta^3(a)$, hence the word $aaa \blacktriangleleft\vartheta^3(a)$ is $\vartheta$-legal.
\end{example}

Let $\#\mc A = d$.
Let $\psi \colon \mc A^\ast \to \N^{d}$ denote the abelianisation function which takes a word $u \in \mc A^\ast$ and enumerates the number of occurrences of each letter in $u$.
That is, $\psi(u) := (|u|_{a_1}, \ldots, |u|_{a_d})$.

\begin{definition}
Let $\vartheta \colon \mc A \to \mc P(\mc A^+)$ be a random substitution.
The substitution $\dsub \colon \mc A \to \mc A^+$ is a \emph{marginal} of $\vartheta$ if $\dsub(a) \in \vartheta(a)$ for every $a \in \mc A$.
A substitution is called \emph{compatible} if, for all $a \in \mc A$, the abelianisation vectors $\psi(\dsub(a))$ are independent of the choice of marginal $\dsub$ of $\vartheta$.
For a compatible substitution $\vartheta$, let $M_\vartheta$ denote the associated \emph{substitution matrix} where the entry $m_{ij}$ is given by $m_{ij}:=|\vartheta(a_j)|_{a_i}$ (which is well-defined by the compatibility of $\vartheta$).
If $M_\vartheta$ is primitive, we let $\lambda := \lambda_{PF}$, the Perron--Frobenius (PF) eigenvalue of $M_\vartheta$, denote the \emph{expansion factor} of $\vartheta$.
\end{definition}
Our use of the term `compatible' differs from previous usage such as in the work of Baake, Spindeler and Stungaru \cite{BSS:random} where they use the term `semi-compatible' and reserve the term `compatible' for a stronger concept which we will later call `strongly compatible'.
We feel the more fundamental property is that the abelianisation is well-defined which is why we have taken the step here to promote semi-compatibility to compatibility and propose this as the new standard.
\begin{example}
The random Fibonacci substitution of Example \ref{EX:fib} has exactly two marginals given by $\dsub_1\colon a \mapsto ab, b \mapsto a$ and $\dsub_2 \colon a \mapsto ba, b \mapsto a$.
The corresponding abelianisation vectors are independent of the chosen marginal and are given by $(|\vartheta(a)|)_{a_j \in \mc A} = (1,1)$ and $(|\vartheta(b)|)_{a_j \in \mc A} = (1,0)$.
The corresponding substitution matrix is given by
\[
M_\vartheta =
\begin{bmatrix}
1 & 1 \\
1 & 0 \\
\end{bmatrix}
\]
with expansion factor $\lambda = \frac{1+\sqrt{5}}{2}$, the golden ratio.

It happens that the subshifts of $\dsub_1$ and $\dsub_2$ are equal.
In such a case where $\vartheta$ is compatible and all marginals of $\vartheta$ have identical subshifts, we call $\vartheta$ \emph{strongly compatible}\footnote{Strongly compatible has previously been referred to as `compatible' in other places \cite{BSS:random}.}.
The random substitution $\vartheta \colon a \mapsto \{aab,baa\}, b \mapsto \{ab\}$ is compatible but is not strongly compatible---this can be seen by noting that the first marginal is conjugate to the square of the Fibonacci substitution, but the second marginal allows for the word $bb$ in its language.
\end{example}

\begin{definition}
The \emph{language} of a random substitution $\vartheta$ is the set of $\vartheta$-legal words,
\[
\mc L_{\vartheta} = \{u \blacktriangleleft \vartheta^k(a) \mid k \geq 0, a \in \mc A  \}.
\]
The set of length-$n$ legal words for $\vartheta$ is denoted by $\mc L^n_\vartheta := \mc L_\vartheta \cap \mc A^n$.
The \emph{random substitution subshift} of $\vartheta$ (\emph{RS-subshift}) is given by
\[
X_{\vartheta} = \{ w \in \mc A^{\Z} \mid u\triangleleft w \Rightarrow u \in \mc L_{\vartheta} \}.
\]
\end{definition}
It is easy to verify that $X_{\vartheta}$ forms a subshift---a closed, shift-invariant subspace of the full shift $\mc A^\Z$.
The following definition is standard and is satisfied by most interesting examples of random substitutions.
\begin{definition}
A random substitution $\vartheta$ on a finite alphabet $\mc A$ is called \emph{primitive} if there exists a $k \in \N$ such that for all $a_i, a_j \in \mc A$ we have $a_i \blacktriangleleft \vartheta^k(a_j)$.
If $\vartheta$ is a primitive random substitution, then we call the associated subshift $X_{\vartheta}$ a \emph{primitive RS-subshift}.
\end{definition}
Recall that a square integer matrix $M$ is called \emph{primitive} if there exists a natural number $k$ such that every entry of $M^k$ is positive.
A compatible RS-subshift is primitive if and only if the associated substitution matrix $M_\vartheta$ is primitive.

Note that, in contrast to the deterministic case, primitivity of a random substitution $\vartheta$ is not enough to conclude that the RS-subshift $X_{\vartheta}$ is non-empty, although it is true that a primitive, compatible random substitution gives rise to a non-empty RS-subshift.

Let $X = (X,\sigma)$ be a subshift on the alphabet $\mc A$.
We say that a point $x \in X$ is \emph{periodic of period} $p \geq 1$ if $\sigma^p(x) = x$ and write $\Per_p(X) = \{x \in X \mid x \text{ has period }p\}$.
Let $\Per(X) = \bigcup_{p \geq 1} \Per_p(X)$.
If $x$ has period $p$ and no smaller periods then we say that the \emph{prime period} of $x$ is $p$ and write $\per(x) = p$.
The \emph{minimal period} of $X$ is the smallest natural number $p_{min}$ such that there exists $x \in X$ with $\per(x) = p_{min}$.

Many results relating to the combinatorics and dynamics of primitive RS-subshifts were presented in previous work of Timo Spindeler and the author \cite{RS:random}.
We briefly mention some of those results that will be useful later.
\begin{prop}[\cite{RS:random}]\label{PROP:rs-subs}
Let $\vartheta$ be a random substitution with associated RS-subshift $X_\vartheta$.
Then:
\begin{itemize}
\item $X_\vartheta$ is closed under substitution.
That is, if $x \in X_\vartheta$, then $y \in \vartheta(x) \implies y \in X_\vartheta$ .
\item $X_\vartheta$ is closed under taking preimages.
That is, if $x \in X_\vartheta$, then there exists an element $y \in X_\vartheta$ and $0 \leq k \leq \max\{|u| : u \in \vartheta(y_0)\} - 1$ such that $\sigma^{-k}(x) \in \vartheta(y)$.
\item If $\vartheta$ is primitive, then the set of periodic elements in $X_\vartheta$ is either empty or dense.
\end{itemize}
\end{prop}

Examples exist of RS-subshifts both with periodic points and without periodic points.
Before moving on to tackling questions about periodic points, we need to introduce and study an important property of random substitutions.

\section{Disjoint Images}\label{SEC:disjoint-images}
A condition called \emph{disjoint sets} for random substitutions was first introduced by Gohlke \cite{G:entropy} where its utility was immediately realised in the calculation of topological entropy for primitive compatible random substitution.
%
We will make use of a slightly stronger related property called the disjoint images property in Sections \ref{SEC:existence-periodic} and \ref{SEC:enumerating-periodic}.
We therefore introduce the property and prove several results allowing one to determine whether a substitution has disjoint images or not.

\begin{definition}
Let $\vartheta$ be a random substitution and let $x \in X_\vartheta$.
We say that a random substitution has \emph{disjoint images} if for all $u, v \in \mc L_\vartheta$,
\[
\vartheta(u) \cap \vartheta(v) \neq \varnothing \implies u = v.
\]
\end{definition}

\begin{example}
Let $\vartheta \colon a \mapsto \{ab,ba\}, b \mapsto \{a\}$ be the random Fibonacci substitution.
We quickly see that $\vartheta$ does not have disjoint images as $\vartheta(ab) \cap \vartheta(ba) = \{aba\} \neq \varnothing$.
\end{example}
The disjoint images condition is stable under taking powers of the substitution.
\begin{prop}
The random substitution $\vartheta$ has disjoint images if and only if $\vartheta^k$ has disjoint images for all $k \geq 1$.
\end{prop}
\begin{proof}
The right-to-left direction is immediate.
Let $\vartheta$ be a random substitution with disjoint images and suppose that $\vartheta^k$ has disjoint images for some $k \geq 1$.
Let $u, v \in \mc L_\vartheta$ be two distinct legal words.
As $\vartheta^k$ has disjoint images, then $\vartheta^k(u) \cap \vartheta^k(v) = \varnothing$.
Let $\tilde{u} \in \vartheta^k(u)$ and $\tilde{v} \in \vartheta^k(v)$, which are both legal words as $u$ and $v$ are legal.
As $\vartheta$ has disjoint images, then $\vartheta(\tilde{u}) \cap \vartheta(\tilde{v}) = \varnothing$.
We have chosen $\tilde{u}$ and $\tilde{v}$ arbitrarily and so, noting that
\[
\bigcup_{\tilde{u} \in \vartheta^k(u)} \vartheta(\tilde{u}) = \vartheta^{k+1}(u) \quad \text{ and } \quad \bigcup_{\tilde{v} \in \vartheta^k(v)} \vartheta(\tilde{v}) = \vartheta^{k+1}(v),
\]
it follows that $\vartheta^{k+1}(u) \cap \vartheta^{k+1}(v) = \varnothing$.
By induction then, $\vartheta^k$ has disjoint images for all $k \geq 1$.
\end{proof}

The disjoint images condition is non-trivial to determine for random substitutions in general.
For random substitutions of constant length however, there is a simple equivalent criterion.

\begin{prop}\label{PROP:const-length-dis-images}
Let $\vartheta$ be a constant length random substitution.
The substitution $\vartheta$ has disjoint images if and only if $\vartheta(a) \cap \vartheta(b) = \varnothing$ for all distinct pairs of letters $a, b \in \mc A$,.
\end{prop}
\begin{proof}
Let $\vartheta$ have constant length $\ell$.
It is clear that if $\vartheta$ has disjoint images, then $\vartheta(a) \cap \vartheta(b) \neq \varnothing$ implies that $a = b$.

For the other direction, let $u, v \in \mc A^+$, suppose $\vartheta(u) \cap \vartheta(v) \neq \varnothing$ and let $w \in \vartheta(u) \cap \vartheta(v)$.
Then, first, as $\vartheta$ has constant length $\ell$, it means that $|u| = |v| = |w|/\ell$.
Consider the word $w_{[\ell i, \ell (i + 1) - 1]}$, which must be an exact inflation word due to its positioning within $w$ and the fact that $\vartheta$ has constant length $\ell$.
It follows that there exists a unique $a \in \mc A$ such that $w_{[\ell i, \ell (i + 1) - 1]} \in \vartheta(a)$ and so the $i$th letter of the preimage of $w$ is uniquely determined for all $i$.
Hence, $u = v$ and so $\vartheta$ has disjoint images.
\end{proof}

\begin{example}\label{EX:rpd}
Let $\vartheta \colon a \mapsto \{ab,ba\}, b \mapsto \{aa\}$ be the \emph{random period doubling substitution}, first studied by Hu, Tian and Wang \cite{HTW:rand-per-doub}.
This substitution has constant length $\ell = 2$.
We note that $\vartheta(a) \cap \vartheta(b) = \varnothing$ and so $\vartheta$ has disjoint images by Proposition \ref{PROP:const-length-dis-images}.
\end{example}

For substitutions of non-constant length, there is a more involved procedure for determining whether it has disjoint images or not as a consequence of the following results.

\begin{lemma}\label{LEM:disj-images-alg}
Let $\vartheta$ be a random substitution for which $u, v \in \vartheta(a) \implies |u| = |v|$ for all $a \in \mc A$.
If $\vartheta$ does not have disjoint images, then there exist distinct letters $a, b \in \mc A$ and $w_a \in \vartheta(a)$, $w_b \in \vartheta(b)$ such that $w_a$ is a prefix of $w_b$.
\end{lemma}
\begin{proof}
As $\vartheta$ does not have disjoint images, let $u, v \in \mc{L}_\vartheta$ be distinct legal words and suppose $w \in \vartheta(u) \cap \vartheta(v)$.
First, note that as $u \neq v$, then $w \in \vartheta(u) \cap \vartheta(v)$ means that $u$ is not a prefix of $v$ or vice versa---this is because if $u$ is a strict prefix of $v$, then $|\vartheta(u)| < |\vartheta(v)|$ by our assumption on all substituted images of letters having the same length, hence $w$ cannot be in both $\vartheta(u)$ and $\vartheta(v)$ at the same time.

As $w \in \vartheta(u)$, there exists a word $w_{u_0} \in \vartheta(u_0)$ such that $w_{u_0}$ is a prefix of $w$.
As $w \in \vartheta(v)$, there exists a word $w_{v_0} \in \vartheta(v_0)$ such that $w_{v_0}$ is a prefix of $w$.
So one of $w_{u_0}$ and $w_{v_0}$ is a prefix of the other.
Either $u_0$ and $v_0$ are not equal, in which case we are done, else $u_0 = v_0$ and so, because every word in $\vartheta(u_0)$ has the same length, we must have $w_{u_0} = w_{v_0}$.

We therefore move to $u_1$ and $v_1$ and repeat the process.
At some point, we must eventually find a position in $u$ and $v$ where they differ because $u$ and $v$ are not equal and $u$ cannot be a prefix of $v$ (or vice versa).
Therefore, at some point we find a position $i$ in $u$ and $v$ where all previous letters $u_j = v_j$ are equal for $j < i$ and map to the same word in $w$, but $u_i \neq v_i$, hence either $w_{u_i} \in \vartheta(u_i)$ is a prefix of $w_{v_i} \in \vartheta(v_i)$ or vice versa.
\end{proof}
Note that if $\vartheta$ is compatible, then it is certainly true that $u, v \in \vartheta(a) \implies |u| = |v|$ for all $a \in \mc A$, and so Lemma \ref{LEM:disj-images-alg} applies to compatible substitutions.
\begin{theorem}\label{THM:disj-images-alg}
Let $\vartheta$ be a random substitution for which $u, v \in \vartheta(a)$ implies that $|u| = |v|$ for all $a \in \mc A$.
There is a decidable procedure to determine if $\vartheta$ has disjoint images or not.
\end{theorem}
\begin{proof}
We should first check whether there exist distinct letters $a, b \in \mc A$ and $w_a \in \vartheta(a)$, $w_b \in \vartheta(b)$ such that $w_a$ is a prefix of $w_b$.
If not, then $\vartheta$ has disjoint images by Lemma \ref{LEM:disj-images-alg}.
Suppose then that such letters $a,b$ and words $w_a, w_b$ exist and collect all such valid letters and words into a quadruple $Q = (a,b;w_a,w_b)$.
The following procedure should be carried out for all valid quadruples.

Let $w'_a$ be the word such that $w_aw'_a = w_b$.
If $w'_a = \varepsilon$, the empty word, then
\[
w_a = w_b \in \vartheta(a) \cap \vartheta(b)
\]
and so $\vartheta$ does not have disjoint images, so we are done.
If $w'_a \neq \varepsilon$, then we should now check if there is some $c \in \mc A$ and $w_c \in \vartheta(c)$ for which $w'_a$ is the prefix of $w_c$ or $w_c$ is the prefix of $w'_a$.
If not, then there is no valid way for a word $w$ to be in the substitutive image of two different words where one begins with $a$ and one begins with $b$, as $a$ cannot be extended to the right, so this quadruple $Q$ does not lead to a pair of words with non-disjoint image.

If $w'_a$ is the prefix of some word $w_c \in \vartheta(c)$ for some $c \in \mc A$ then let $w''_a$ be the word such that $w'_a w''_a = w_c$ (we must check all possible valid pairs $(c,w_c)$ for which this occurs).
Again, either $w''_a = \varepsilon$ in which case $w_b = w_a w'_a = w_a w'_a w''_a = w_a w_c  \in \vartheta(b) \cap \vartheta(ac)$ and so if $ac \in \mc L_\vartheta$, then $\vartheta$ does not have disjoint images, or else we should check if there is some $d \in \mc A$ and $w_d \in \vartheta(d)$ for which $w''_a$ is the prefix of $w_d$ or $w_d$ is the prefix of $w''_a$.

If some word $w_c \in \vartheta(c)$ is the prefix of $w'_a$ for some $c \in \mc A$ then let $w'_c$ be the word such that $w_c w'_c = w'_a$ and again perform the above checks in the same way.

By repeating the above procedure, we either show that $Q$ does not lead to a pair of legal words with disjoint image, or it does, or we enter a loop, because there are only finitely many subwords $w^{(n)}_a$ of inflation words in $\vartheta(\mc A)$.
If we enter such a loop and return to having to check for the same subword of an inflation word again without finding a pair of words with disjoint image, then we know that $Q$ will never lead to such a pair.
\end{proof}
%
We illustrate the above algorithm with two examples.
\begin{example}
Let $\vartheta$ be a compatible random substitution on $\{0,1,2\}$ given by 
\[
\vartheta \colon 0 \mapsto \{0102, 1200, 0012\}, 1 \mapsto \{010\}, 2 \mapsto \{20102010\}.
\]
We notice that $w_1 = 010 \in \vartheta(1)$ is a prefix of $w_0 = 0102 \in \vartheta(0)$ and there are no other valid prefixes.
Our only quadruple to check is then $Q = (1, 0; 010, 0102)$.
Setting $w'_1 = 2$, we have $w_1 w'_1 = w_0$ and $w'_1$ is a prefix of $w_2 = 20102010 \in \vartheta(2)$.
Setting $w'_2 = 0102010$, we have that $w'_2$ is not the prefix of any $w_c \in \vartheta(\mc A)$, however $w''_0 = 0102 \in \vartheta(0)$ is a prefix of $w'_2$ leaving the remainder $w''_2 = 010$ as a suffix of $w_2$ which is in $\vartheta(1)$ with no remainder left over.

It follows that $01020102010 \in \vartheta(12) \cap \vartheta(001)$.
As both $12$ and $001$ are legal words for $\vartheta$, $\vartheta$ does not have disjoint images.
\end{example}
\begin{example}
Let $\vartheta$ be a compatible random substitution on $\{0,1\}$ given by 
\[
\vartheta \colon 0 \mapsto \{010, 100\}, 1 \mapsto \{0101\}.
\]
We notice that $w_0 = 010 \in \vartheta(0)$ is a prefix of $w_1 = 0101 \in \vartheta(1)$ and there are no other valid prefixes.
Our only quadruple to check is then $Q = (0, 1; 010, 0101)$.
Setting $w'_0 = 1$, we have $w_0 w'_0 = w_1$ and $w'_0$ is a prefix of $w''_0 = 100 \in \vartheta(0)$.
Setting $w'''_0 = 00$, we have that $w'''_0$ is not the prefix of any $w_c \in \vartheta(\mc A)$ and no $w_c \in \vartheta(\mc A)$ is a prefix of $w'''_0$.
It follows that the quadruple $Q$ does not lead to any non-disjoint images and so $\vartheta$ has disjoint images by Theorem \ref{THM:disj-images-alg}.
\end{example}

There is a close link between the disjoint images property and the classical notion of unique (global) recognisability \cite{M:aperiodic}.
Recall the definition of unique recognisability for a deterministic substitution:

Let $\phi$ be a deterministic substitution and let $X_\phi$ be the associated subshift for $\phi$.
We say that $\phi$ is \emph{globally uniquely recognisable} if for all $x \in X_\phi$, there exists a unique $y \in X_\phi$ such that $\sigma^{-k}(x) = \phi(y)$ and $0 \leq k \leq |\phi(y_0)| - 1$.

For $\phi$ a primitive deterministic substitution, it was shown by Moss\'{e} \cite{M:aperiodic} that if $X_\phi$ is aperiodic (has infinitely many elements) then $\phi$ is globally uniquely recognisable (the converse is very simple---see Proposition \ref{PROP:recog}).
We can extend the definition of global unique recognisability to compatible random substitutions in a natural way.
\begin{definition}
Let $\vartheta$ be a compatible random substitution.
We say that $\vartheta$ is \emph{globally uniquely recognisable} if for all $x \in X_\vartheta$, there exists a unique $y \in X_\phi$ such that $\sigma^{-k}(x) \in \vartheta(y)$ and $0 \leq k \leq |\vartheta(y_0)| - 1$.
\end{definition}
\begin{example}
Let $\vartheta$ be the compatible random substitution on $\{a,b\}$ given by
\[
\vartheta \colon a \mapsto \{abbabba, ababbba\}, b \mapsto \{a\}.
\]
For any element $x$ of $X_\vartheta$, any subword $u \triangleleft x$ of the form $u = abbabba$ must have come from an $a$ in a preimage of $x$ as there is no other concatenation of inflation words that contains $u$ as a subword.
Similarly, if $u$ is of the form $u = ababbba$ then $u$ must also have come from an $a$ in a preimage of $x$.
This identifies the exact positions of all $a$s in a preimage of $x$ and every other position in $x$ must hence be filled by a $b$ (corresponding to subwords in $x$ which are of the form $a$ and such that its immediate left and right neighbours are also $a$s).
It follows that there is only one preimage of $x$ up to the least positive shift placing the origin at the intersection of two inflation words.
So $\vartheta$ is globally uniquely recognisable.
\end{example}
The next result is clear and provided without proof.
\begin{prop}\label{PROP:recog-powers}
The compatible random substitution $\vartheta$ is uniquely globally recognisable if and only if for all $m \geq 1$, the compatible random substitution $\vartheta^m$ is uniquely globally recognisable.
\end{prop}
The following argument is essentially identical to the one for deterministic substitutions.
\begin{prop}\label{PROP:recog}
Let $\vartheta$ be a primitive compatible random substitution.
If $\vartheta$ is globally uniquely recognisable, then $X_\vartheta$ contains no periodic points.
\end{prop}
\begin{proof}
Suppose that $\Per(X_\vartheta)$ is non-empty and let $x \in \Per(X_\vartheta)$ be periodic of period $p$.
Let $m \geq 1$ be large enough so that $|\vartheta^m(a)| \geq p$ for all $a \in \mc A$.
Such an $m$ exists by primitivity and compatibility of $\vartheta$.
We know that there exists an element $y \in X_\vartheta$ such that $\sigma^{-k}(x) \in \vartheta^m(y)$ where $0 \leq k \leq |\vartheta^k(y_0)| - 1$ by Proposition \ref{PROP:rs-subs}.
It follows that $\sigma^{|\vartheta^m(y_0)| - k}(x) \in \vartheta^m(\sigma(y))$.
As $x$ is $p$-periodic, then for all $i \in \Z$, we also have $\sigma^{|\vartheta^m(y_0)| - k - ip}(x) \in \vartheta^m(\sigma(y))$.

Let $i_0$ be the least integer so that $k_0 := - |\vartheta^m(y_0)| + k + i_0 p$ is non-negative.
Hence, $0 \leq - |\vartheta^m(y_0)| + k + i_0 p = k_0$.
Suppose that $k_0 \geq |\vartheta^m(y_1)|$.
Then $k_0 \geq p$ by our choice of $m$, so $k_0 - p = |\vartheta^m(y_0)| + k + (i_0-1) p \geq 0 $ but then $i_0$ is not minimal, contradicting the choice of $i_0$.
It follows that
\[
0 \leq k_0 \leq |\vartheta^m(y_1)| - 1
\]
and so we have both $\sigma^{k_0}(x) \in \vartheta^m(\sigma(y))$ with $0 \leq k_0 \leq |\vartheta^m(y_1)| - 1 = |\vartheta^m(\sigma(y)_0)| - 1$ and $\sigma^{-k}(x) \in \vartheta^m(y)$ with $0 \leq k \leq |\vartheta^m(y_0)|-1$.
It is clear that $y \neq \sigma(y)$ by primitivity, compatibility and the fact that $\#\mc A \geq 2$.
Hence, $y$ is not unique.
It follows that $\vartheta^m$ is not globally uniquely recognisable and so $\vartheta$ is not globally uniquely recognisable by Proposition \ref{PROP:recog-powers}.
\end{proof}

Let us also define a slightly weaker notion of disjoint images.
\begin{definition}
Let $\vartheta$ be a compatible random substitution.
We say that $\vartheta$ has \emph{disjoint inflation images} if for all $a \in \mc A$, $m \geq 1$ and $u, v \in \vartheta^m(a)$, then $\vartheta(u) \cap \vartheta(v) \neq \varnothing \implies u = v$.
\end{definition}
We have the following result linking the two notions.
\begin{prop}\label{PROP:recog-disj-inflation}
Let $\vartheta$ be primitive and compatible.
If $\vartheta$ is globally uniquely recognisable, then $\vartheta$ has disjoint inflation images.
\end{prop}
\begin{proof}
Let us prove the contrapositive.
Suppose that $\vartheta$ does not have disjoint inflation images, so there exists a letter $a \in \mc A$, $m \geq 1$ and $u, v \in \vartheta^m(a)$ such that $\vartheta(u) \cap \vartheta(v) \neq \varnothing$.
Let $w$ be an element in the non-empty intersection $\vartheta(u) \cap \vartheta(v)$.
As $\vartheta$ is primitive, $X_\vartheta$ is either empty, in which case the result is vacuously true, or there exists an element $x \in X_\vartheta$ which contains the element $a$.
We may suppose that such an element has $a$ occurring in the position $x_0 = a$.
Call such an element $\hat{x}$ (so $\hat{x}_0 = a$) and construct new elements $x_1$ and $x_2$ in $\vartheta^m(\hat{x})$ such that every realisation of $\vartheta^m$ on the letters in $\hat{x}$ is chosen identically except for at the position $\hat{x}_0$, where for $x_1$ we choose the realisation $u \in \vartheta^m(\hat{x}_0)$ and for $x_2$ we choose the realisation $v \in \vartheta^m(\hat{x}_0)$.
So $x_1$ and $x_2$ agree on every position except that $(x_1)_{[0,|\vartheta(a)| - 1]} = u$ and $(x_1)_{[0,|\vartheta(a)| - 1]} = v$.
Note that $x_1$ and $x_2$ are legal elements in $X_\vartheta$ as $\hat{x}$ is in $X_\vartheta$ and $X_\vartheta$ is closed under taking images of $\vartheta$.

Now choose an element $x_{3} \in \vartheta(x_1) \subset X_\vartheta$ such that the realisation of $\vartheta$ on the word $(x_1)_{[0, \vartheta(a)| - 1]} = u$ is $w \in \vartheta(u)$.
From construction we see that $x_3$ is also in $\vartheta(x_2)$ because $x_1$ and $x_2$ coincide outside of the range $[0, |\vartheta(a)|-1]$ and $w$ is in both $\vartheta((x_1)_{[0, |\vartheta(a)|-1]}) = \vartheta(u)$ and $\vartheta((x_2)_{[0, |\vartheta(a)|]}) = \vartheta(v)$.
It follows that $\vartheta$ is not globally uniquely recognisable.
\end{proof}
As a consequence, we immediately see that the random Fibonacci substitution is not globally uniquely recognisable as $\vartheta(ab) \cap \vartheta(ba) = \{aba\}$ and $ab, ba \in \vartheta(a)$ so it does not have disjoint inflation images.

The converse of Proposition \ref{PROP:recog-disj-inflation} is not true. A counterexample is given by the random period doubling substitution which has periodic points in its subshift and so is not globally uniquely recognisable by Proposition \ref{PROP:recog}. However, the random period doubling substitution has disjoint images (hence disjoint inflation images) as we saw in Example \ref{EX:rpd}.

\section{Existence of Periodic Points}\label{SEC:existence-periodic}

We say that a word $u = u_1 \cdots u_p$ is a \emph{periodic block} for a periodic element $x \in \Per(X)$ if for some $i\in \mathbb{Z}$, $x_{i+j+np} = u_j$ for all $n \in \Z$ and $ 1 \leq j \leq p$.
We say that $u$ is a \emph{prime periodic block} for $x$ if $u$ is a periodic block for $x$ and $|u| = \per(x)$.
For a primitive matrix $M$, let $\boldsymbol{R} = (\boldsymbol{R}_1, \ldots, \boldsymbol{R}_d)^T$ be its right Perron--Frobenius (PF) eigenvector normalised so that $\|\boldsymbol{R}\|_1 = 1$ and, if the entries of $\boldsymbol{R}$ are rationally related, let $\hat{\boldsymbol{R}} = (\hat{\boldsymbol{R}}_1, \ldots, \hat{\boldsymbol{R}}_d)^T$ be its right PF eigenvector normalised so that all entries are positive natural numbers and minimal.

\begin{lemma}\label{PROP:periodic-block-letter-counts}

Let $x \in \Per(X_\vartheta)$ and let $u$ be a periodic block for $x$ of length $p$.
For all $a_i \in \mc A$, $|u|_{a_i} = p\|\hat{\boldsymbol{R}}\|_1^{-1}\hat{\boldsymbol{R}}_i$.
Further, $p\|\hat{\boldsymbol{R}}\|_1^{-1}$ is an integer.
\end{lemma}
\begin{proof}
The word $u$ is a periodic block for $x$ and so without loss of generality, we may assume that $x = \cdots uu.uu \cdots$ for some $i \geq 0$.
Note that the proportion (with respect to the total length) of the letter $a_i$ in the word $x_{[-jp,jp]} = u^{2jp}$ for $j \geq 1$ is constant.
That is,
\[
\frac{|x_{[-jp,jp]}|_{a_i}}{|x_{[-jp,jp]}|_{\hphantom{a_i}}} = \frac{|u^{2j}|_{a_i}}{2jp} = \frac{2j|u|_{a_i}}{2jp} = \frac{|u|_{a_i}}{p}.
\]
As $x$ is periodic, its letter frequencies are well-defined (in fact this is true for every element in $X_\vartheta$ as $\vartheta$ is compatible).
It follows that the frequency of the letter $a_i$ in $x$ is given by
\[
\freq_{a_i}(x) = \lim_{j \to \infty}\frac{|x_{[-jp,jp]}|_{a_i}}{|x_{[-jp,jp]}|_{\hphantom{a_i}}} = \frac{|u|_{a_i}}{p}.
\]

As in the case of deterministic substitutions, because we are only concerned with abelianisations, letter frequencies of all points in $X_\vartheta$ are equal to the corresponding entries of the normalised right PF eigenvector $\boldsymbol{R} = (\boldsymbol{R}_1, \ldots, \boldsymbol{R}_d)$ of $M$ \cite[Sec 5.4]{Q:book} given by $\boldsymbol{R}_i = \hat{\boldsymbol{R}}_i/\|\hat{\boldsymbol{R}}\|_1$.

So we have
\[
\frac{|u|_{a_i}}{p} = \boldsymbol{R}_i = \frac{\hat{\boldsymbol{R}}_i}{\|\hat{\boldsymbol{R}}\|_1} \implies |u|_{a_i} = \frac{p\hat{\boldsymbol{R}}_i}{\|\hat{\boldsymbol{R}}\|_1}.
\]
It remains to show that $p/\|\hat{\boldsymbol{R}}\|_1$ is an integer.
As $\|\hat{\boldsymbol{R}}\|_1$ is a positive integer, write $p/\|\hat{\boldsymbol{R}}\|_1:= k/m$ in lowest terms.
So
\[
|u|_{a_i} = \frac{k\hat{\boldsymbol{R}}_i}{m}
\]
and as the left hand side is an integer, so is the right.
It follows that $m$ divides $k\hat{\boldsymbol{R}}_i$, but $k/m$ is in lowest terms and so $m\mid \hat{\boldsymbol{R}}_i$ for all $1 \leq i \leq d$.
As the normalisation $\hat{\boldsymbol{R}}$ was chosen so that all entries are positive natural numbers and minimal, it follows that $m=1$ and hence $p/\|\hat{\boldsymbol{R}}\|_1 = k$ is an integer.
\end{proof}
We let $\mathfrak{p} := \|\hat{\boldsymbol{R}}\|_1$ denote the \emph{virtual period} of $\vartheta$; so-called because of the following immediate consequence of Lemma \ref{LEM:preimages}.
\begin{coro}\label{COR:divisible}
If $x \in \Per(X_\vartheta)$, then $\mathfrak{p}$ divides $\per(x)$.
\end{coro}
Note that $\mathfrak{p}$ is the virtual period of $\vartheta$ if and only if $\mathfrak{p}$ is the virtual period of $\vartheta^k$ for all $k \geq 1$.

\begin{prop}\label{PROP:rational}
Let $\vartheta$ be a primitive compatible random substitution with $\Per(X_\vartheta) \neq \varnothing$.
Then we have $\boldsymbol{R}_i \in \Q$ for all $1 \leq i \leq d$.
Further, $\lambda$ is a positive integer.
\end{prop}
\begin{proof}
Let $1 \leq i \leq d$.
Suppose that $\Per (X_\vartheta) \neq \varnothing$ and let $x \in \Per_p(X_\vartheta)$ be a periodic point with periodic block $u = (u_1 \cdots u_p)$.
As before, by compatibility of $\vartheta$, it is well known that letter frequencies of all points in $X_\vartheta$ are uniformly well-defined and equal to the corresponding entries of the normalised right PF eigenvector of $M$.
As $x$ is periodic, we know that for all $i$, the frequency of the letter $a_i$ in $x$ is exactly $|u|_{a_i}/p$.
It follows that $\boldsymbol{R}_i = |u|_{a_i}/p$ and, as $p$ is a non-zero integer, it follows that $\boldsymbol{R}_i \in \Q$.

Considering both sides of the equation $M\hat{\boldsymbol{R}} = \lambda \hat{\boldsymbol{R}}$.
All entries of the left hand side are non-trivial, non-negative, linear combinations of positive integers and so for all $1 \leq i \leq d$, $(M\hat{\boldsymbol{R}})_i$ is a positive integer.
Suppose $\lambda \in \Q \setminus \N$.
So $\lambda = p/q$ for $\gcd(p,q) = 1$ and $q \geq 2$.
It follows that $\frac{p}{q} (\hat{\boldsymbol{R}})_i$ is an integer for every $i$ and so $p (\hat{\boldsymbol{R}})_i$ is divisible by $q$ for every $i$.
As $p$ and $q$ are coprime, it follows that $(\hat{\boldsymbol{R}})_i$ is divisible by $q \geq 2$ for every $i$ which contradicts the construction of $\hat{\boldsymbol{R}}$ being the minimal positive integer normalisation of $\boldsymbol{R}$.
It follows that $\lambda$ is a positive integer.
\end{proof}

This result gives a necessary criterion for the existence of periodic points in the RS-subshift: If $\lambda$ is not an integer, then $\Per(X_\vartheta) = \varnothing$.

\begin{example}
Let $\vartheta \colon a \mapsto \{ab, ba\}, b \mapsto \{a\}$ be the random Fibonacci substitution which is compatible.
The associated substitution matrix is given by $M =
\left[\begin{smallmatrix}
1 & 1 \\
1 & 0 \\
\end{smallmatrix}\right]
$ whose PF eigenvalue is $\lambda = \frac{1 + \sqrt{5}}{2}$, the golden ratio.
As $\lambda$ is not an integer, it follows from Proposition~\ref{PROP:rational} that $\Per(X_\vartheta) = \varnothing$.
\end{example}

The converse is certainly false, as evidenced by any aperiodic deterministic substitution with integer expansion factor.

\begin{lemma}\label{LEM:preimages}
If $x$ is a periodic element in $X_\vartheta$ with prime period $p$, then there exists a prime periodic block $u$ of $x$ and a natural number $n$ such that $u^n$ is a concatenation of exact inflation words $u^n \in \vartheta(a_1) \cdots \vartheta(a_k)$ with $a_1 \cdots a_k \in \mc L_\vartheta$ and $k \leq p$.
\end{lemma}
\begin{proof}
First, as $x \in X_\vartheta$ then there exists an element $y \in X_\vartheta$ such that some shift of $x$ is in $\vartheta(y)$.
In particular, without loss of generality, we may assume that there exists a decomposition of $x$ into exact inflation words
\[
x = \cdots \vartheta(y_{-1}). \vartheta(y_0) \vartheta(y_1) \cdots.
\]
Let $\tilde{u}$ be some prime periodic block of $x$.
The left-most endpoint of the inflation word $\vartheta(y_0)$ coincides with some position $(\tilde{u})_i$ within a copy of $\tilde{u}$ where $0 \leq i \leq p-1$.
Similarly, $\vartheta(y_j)$ coincides with one of only finitely many possible positions in a copy of $\tilde{u}$.
In particular, for the consecutive occurrence of inflation words $\vartheta(y_0) \cdots \vartheta(y_{p})$, at least one position with the same index in two separate copies of $\tilde{u}$ coincide with left endpoints of a corresponding inflation word by an application of the pigeon-hole principle.
Let $\vartheta(y_{j})$ and $\vartheta(y_{k})$ with $j < k$ be such an occurrence and suppose their left endpoints coincide with a copy of $\tilde{u}$ at position $0 \leq i \leq p$.
Let $u = \alpha^i(\tilde{u})$ be the $i$-fold cyclic permutation of $\tilde{u}$.
Then the left endpoints of $\vartheta(y_j)$ and $\vartheta(y_k)$ coincide with the left endpoints of different copies of $u$ and hence also the right endpoint of $\vartheta(y_{k-1})$ coincides with the right endpoint of a copy of $u$ which is never to the left (it may be the same as) the right endpoint of the copy of $u$ whose left endpoint coincides with the left endpoint of $\vartheta(y_j)$.

It follows that $\vartheta(y_j)\cdots \vartheta(y_{k-1})$ is exactly $u^n$ for some $n$.
Note that $y \in X_\vartheta$ and so the subword $y_j \cdots y_{k-1}$ is legal, and as both $j$ and $k-1$ are bounded between $0$ and $p-1$, we also have $|y_j y_{j+1}\cdots y_{k-1}| \leq p$ as required.
\end{proof}

Lemma \ref{LEM:preimages} allows us to show the existence of periodic preimages of periodic words if we assume $\vartheta$ has disjoint images.

\begin{prop}\label{PROP:periodic-preimages}
Let $\vartheta$ be a random substitution with disjoint images and let $x \in \operatorname{Per}(X_\vartheta)$.
If the prime period of $x$ is $p$ then for some $i \in \Z$ there exists a point $y \in \Per(X_\vartheta)$ with $\sigma^i(x) \in \vartheta(y)$ and $\per(y) \leq p$.
\end{prop}
\begin{proof}
Let $y$ be a preimage of some shift of $x$ under $\vartheta$.
By Lemma \ref{LEM:preimages}, there exists a prime periodic block $u = x_{[j,j+p-1]}$ of $x$ for some $j \in \Z$ and a natural number $n$ such that $u^n$ is an exact concatenation of inflation words $u^n \in \vartheta(a_1) \cdots \vartheta(a_k)$ and $k \leq p$.
Note also that in the proof of Lemma \ref{LEM:preimages}, the point $y$ included this word $a_1 \cdots a_k = y_{[s,t]}$ as a subword which exactly mapped to $u^n$ under $\vartheta$.
As $\vartheta$ has disjoint images and $u^n$ is a concatenation of exact inflation words, the legal preimage of the word $x_{[j,j+np-1]}$ in $y$ is uniquely determined and must be $a_1 \cdots a_k$.
It follows that every subword of $x$ of the form $x_{[j+knp,j+(k+1)np-1]}$ for some $k \in \Z$ also has a uniquely determined preimage in $y$ and it must also be the word $a_1 \cdots a_k$.
As these words fully cover $x$, it follows that $y$ is (a shift of) the bi-infinite sequence $\cdots (a_1\cdots a_k).(a_1 \cdots a_k) \cdots$ whose prime period is at most $k \leq p$.
\end{proof}
\begin{remark}
Note that it is possible for the prime period of $y$ to be equal to $p$ as the periods of periodic preimages cannot decrease indefinitely.
Hence the bound on the prime period of $y$ is sharp.
Concretely, we have the example of the word $(aab)^\infty$ for the random period doubling substitution whose only preimages are shifted copies of itself.

\end{remark}

We do not know how to remove the disjoint images assumption from the statement of Proposition \ref{PROP:periodic-preimages}, but it is possible that the assumption is unnecessary.
It is not at all clear why a periodic element in the RS-subshift could not otherwise have only non-periodic preimages.
It is certainly true that non-periodic preimages can exist for periodic elements in $X_\vartheta$.

For completeness, we should also discuss the existence of periodic \emph{images}, though these are much simpler to identify.
\begin{prop}
If $x \in \Per_p(X_\vartheta)$ then there exists $y \in \vartheta(x)$ such that $y \in \Per_{\lambda p}(X_\vartheta)$, where $\lambda$ is the PF eigenvalue of $M$.
\end{prop}
\begin{proof}
Recall that $\hat{\boldsymbol{R}}$ is the right PF eigenvector of $M$ normalised so that all entries are positive integers and minimal.
Let $u$ be a periodic block for $x$ of length $p$, so $x = \cdots uu.uu \cdots$.
By Lemma \ref{PROP:periodic-block-letter-counts}, there exists $k \geq 1$ such that for all $a_i \in \mc A$, $|u|_{a_i} = k \hat{\boldsymbol{R}}_i$.
Let $v \in \vartheta (u)$.
As $\vartheta$ is compatible,
\[
|v|_{a_i} = |\vartheta(u)|_{a_i} = (M(k \hat{\boldsymbol{R}}_1, \ldots, k \hat{\boldsymbol{R}}_d)^T)_i = k (M \hat{\boldsymbol{R}})_i = k \lambda \hat{\boldsymbol{R}}_i.
\]
It follows that $|v| = \sum_{i=1}^d |v|_{a_i} = \sum_{i=1}^d k \lambda \hat{\boldsymbol{R}}_i = \lambda \sum_{i=1}^d |u|_{a_i} = \lambda |u| = \lambda p$.
Hence, by taking the realisation of $\vartheta$ on $x$ which maps every periodic block $u$ to $v$, the resulting sequence $y \in \vartheta(x)$ has periodic block $v$ of length $\lambda p$ and so $y \in \Per_{\lambda p}(X_\vartheta)$.
\end{proof}

Proposition \ref{PROP:rational} is only a necessary criterion for the existence of periodic points.
We would prefer to have stronger criteria which can be applied when $\lambda$ is an integer, such as when $\vartheta$ is of constant length.

\begin{lemma}\label{LEM:main}
Let $\vartheta$ be a primitive, compatible random substitution of constant length $\ell$.
Let $v$ be a legal word of length $|v| = \ell$ such that $v \notin \vartheta(\mc A)$.
If $\gcd(p,\ell)=1$, then no element in $\Per_p(X_\vartheta)$ can contain $v$ as a subword.
\end{lemma}
\begin{proof}
Let $x\in \Per_p(X_\vartheta)$ and let $v$ be a subword of $x$ of length $|v| = \ell$.
Without loss of generality, suppose that $x \in \vartheta(y)$ for some $y \in X_\vartheta$, so that each word of the form $x_{[i\ell, (i+1)\ell -1]}$ is an exact inflation word.
Suppose that $x_{[j,j+\ell-1]} = v$ is the first appearance of $v$ in $x$ immediately to the right of the origin.
By the $p$-periodicity of $x$, $v$ also appears at $x_{[j+np, j+np+\ell-1]}$ for every $n \in \Z$.
By B\'{e}zout's identity, let $s, t \in \Z$ be such that $sp + t \ell = 1$.
At $n = -js$ then, we have
\[
j+np = j - jsp = j - j(1 - t \ell) = jt\ell
\]
and so $x_{[j t \ell, (jt + 1) \ell - 1]} = v$.
However, we know that every word of the form $x_{[i\ell, (i+1)\ell -1]}$ for $i \in \Z$ is an exact inflation word, and so $v \in \vartheta(\mc A)$.
It follows that any $x \in \Per_p(X_\vartheta)$ can only contain subwords of length $\ell$ which are inflation words.
\end{proof}

\begin{theorem}\label{THM:empty-periodic}
Let $\vartheta$ be a primitive, compatible random substitution of constant length $\ell$ with disjoint images.
If there exists a legal word $v$ of length $|v| = \ell$ such that $v \notin \vartheta(\mc A)$ and which appears in every element of $X_\vartheta$, then $\Per(X_{\vartheta}) = \varnothing$.
\end{theorem}
\begin{proof}
Let $x \in \Per_p(X_\vartheta)$ be a periodic element of prime period $p$ with prime periodic block $u$.
Suppose $\gcd(p,\ell) = d \geq 1$.
If $d = 1$, then we know that $x$ cannot contain $v$ as a subword by Theorem \ref{LEM:main}.
So we reach a contradiction to the assumption that every element of $X_\vartheta$ contains $v$ as a subword.
So suppose $d \geq 2$.
Write $p':= p/d$ and $\ell' := \ell/d$.
Note that $p' < p$ because $d \geq 2$.
Let $u' = u^{\ell'}$ which is a periodic block for $x$ of length
\[
|u'| = |u|\ell' = p\ell' = p'd\frac{\ell}{d} = p'\ell.
\]
As $x$ is in $X_\vartheta$, there exists an integer $i$ and an element $y \in X_\vartheta$ such that $\sigma^i(x) \in \vartheta(y)$.
We can then find a cyclic permutation $\beta(u')$ of the word $u'$ such that $\beta(u')$ is a concatenation of exact inflation words
\[
\beta(u') \in \vartheta(y_m) \cdots \vartheta(y_{m + p' - 1})
\]
because $|u'|$ has length $p'\ell$ and so is a multiple of $\ell$.

The word $w = y_{[m, m + p' - 1]}$ has length $p'$ and since $\vartheta$ has disjoint images and the periodic element $\sigma^i(x)$ is in $\vartheta(y)$, we know that $\sigma^j(y) = w^{\infty}$ for some shift $j \in \Z$, so $y$ has $p'$ as a period.
If $\gcd(p',l) = 1$ then by again applying Theorem \ref{LEM:main}, we may conclude that $y$ cannot contain $v$ as a subword and so we reach a contradiction.
It follows that $\gcd(p',\ell) = d' \geq 2$ and so we may repeat this procedure by defining $p'' := p'/d'$ and $\ell'' = p'/d'$ and in general $d^{(n)} = \gcd(p^{(n)}, \ell)$, $p^{(n+1)} := p^{(n)}/d^{(n)}$ and $l^{(n+1)} := l^{(n)}/d^{(n)}$.
Note however that, as $d^{(n)} \geq 2$ and $p^{(n)} > 0$ for all $n$, we have
\[
p > p' > p'' > \cdots > p^{(n)} > \cdots > 0
\]
which cannot happen for infinitely many $p^{(n)}$s and so at some point there must exist $N$ such that $d^{(N)} = 1$.
Hence, there exists some point $y^{(N)} \in X_\vartheta$ with period $p^{(N)}$ such that $\gcd(p^{(N)},\ell) = 1$ and hence cannot contain $v$ by Theorem \ref{LEM:main}.
We picked $x$ to be an arbitrary periodic element in $X_\vartheta$ and so it follows that no periodic elements in $X_\vartheta$ can exist.
Hence, $\Per(X_\vartheta) = \varnothing$.
\end{proof}
\begin{remark}
Note that the condition that every element of $X_\vartheta$ contains $v$ as a subword is necessary.
Even if almost every element contains $v$ as a subword, but there exists an element which has no appearance of $v$, then periodic elements can exist.
For example, if $\vartheta$ is the random period doubling substitution, even though $\vartheta$ is constant length with disjoint images and $bb \notin \vartheta(\mc A)$, the word $bb$ does not appear as a subword of every element of $X_\vartheta$, and hence Theorem \ref{THM:empty-periodic} does not apply.
So the periodic element $(aab)^\infty$ is free to exist (which it does).
\end{remark}

\begin{example}
Let $\vartheta \colon a \mapsto \{aabba, ababa\}, b \mapsto \{aaaaa\}$.
This is a primitive compatible random substitution of constant length $\ell = 5$ and $\vartheta$ has disjoint images by Proposition \ref{PROP:const-length-dis-images}.
Note that every element of $X_\vartheta$ clearly contains the word $ba$ because $X_\vartheta$ is non-empty and a transition from $b$s to $a$s must occur in every element as the longest legal string of consecutive $a$s in a row is $13$ and the longest legal string of consecutive $b$s in a row is $2$.
It follows that every element of $x$ must contain the word $aaaaaaabba \in \vartheta(ba)$ or the word $aaaaaababa \in \vartheta(ba)$ because $x$ has a legal preimage in $X_\vartheta$ which contains the word $ba$, hence $x$ contains a subword of the form $\vartheta(ba)$.
It follows that $x$ always contains the word $aaaab$ which has length $5$ and is not itself an inflation word.
Hence, by Corollary \ref{THM:empty-periodic}, $\Per(X_\vartheta) = \varnothing$.

Note that $\vartheta$ also happens to be globally uniquely recognisable and so we could have also shown that $\Per(X_\vartheta) = \varnothing$ using Proposition \ref{PROP:recog}.
\end{example}
We can slightly improve upon Theorem \ref{THM:empty-periodic} by allowing for a possible set of `unavoidable words'.
The proof is essentially identical.
\begin{coro}
Let $\vartheta$ be a primitive, compatible random substitution of constant length $\ell$ with disjoint images.
If every element of $X_\vartheta$ contains a word from the set $\mc L^\ell_{\vartheta} \setminus \vartheta(\mc A)$ as a subword, then $\Per(X_\vartheta) = \varnothing$.
\end{coro}

\section{Enumerating periodic points}\label{SEC:enumerating-periodic}
Lemma \ref{LEM:preimages} and Proposition \ref{PROP:periodic-preimages} are the key to building an algorithm for testing if a particular word is a periodic block for an element of an RS-subshift $X_\vartheta$---ultimately allowing us to count the number of periodic points in $X_\vartheta$ with a specific period $p$.
In particular, given a compatible random substitution $\vartheta$ with disjoint images and a test word $u \in \mc A^n$, we can check whether $u$ is a periodic block for an element in $X_\vartheta$ using the following result.

\begin{theorem}\label{THM:enumeration-alg}
Let $\vartheta$ be a compatible random substitution with disjoint images.
There is a decidable procedure to determine whether the word $u$ is a prime periodic block for an element of $\Per(X_\vartheta)$.
\end{theorem}
\begin{proof}
Let $u \in \mc A^p$ be a test word.
Without loss of generality, we may assume that $u \neq v^k$ for all words $v$ and $k \geq 2$.
Set $u_0 := u$.
If $(u_0)^\infty \in \Per(X_\vartheta)$ then $u_0 \in \mc L^p_\vartheta$.
It is easy to check if a particular word is legal for the substitution and so we reject $u_0$ if it is not legal, else we move to the next step.

Let $|\vartheta|$ denote the length of the longest inflation word for the substitution, $|\vartheta| := \max\{|u| : u \in \vartheta(\mc A)\}$.
If $(u_0)^\infty \in \Per(X_\vartheta)$ then by Lemma \ref{LEM:preimages}, there exists a cyclic permutation $\alpha^i$ and a natural number $n$ such that $\alpha^i((u_0)^n)$ is a concatenation of exact inflation words $\alpha^i((u_0)^n) \in \vartheta(a_1) \cdots \vartheta(a_k)$ with $a_1 \cdots a_k \in \mc L_\vartheta$ and $k \leq p$.
Note that $|\alpha^i((u_0)^n)| = np$ and
\[
|\vartheta(a_1) \cdots \vartheta(a_k)| \leq |\vartheta|k \leq |\vartheta|p
\]
and so we know that $n$ is bounded above by $|\vartheta|$, so we only need to check the words $\alpha^i((u_0)^j)$ for each $1 \leq j \leq |\vartheta|$ and $0 \leq i \leq p-1$ to see if any are legal concatenations of inflation words.
If no such legal concatenation of inflation words exists, then $(u_0)^\infty$ cannot be a legal periodic element of $X_\vartheta$.

Otherwise, for each pair $(i,j)$ there exists at most one word $u_{1,i,j}$ (by disjoint images) such that $\alpha^i((u_0)^j) \in \vartheta(u_{1,i,j})$ where $u_{1,i,j} \in \mc L^k_\vartheta$ for $k \leq p$ and by Proposition \ref{PROP:periodic-preimages}, if $(u_0)^\infty \in \Per(X_\vartheta)$, then the words $u_{1,i,j}$ are also periodic blocks for periodic elements in $\Per(X_\vartheta)$.

We repeat this procedure, building words $u_{m,i,j}$ which are to be tested if they are valid candidates for being periodic blocks.
As $|u_{m,i,j}| \leq p$ for all $m,i,j$, it follows that there are only finitely many candidate words and so either we run out of valid preimages, in which case $u$ is not a prime periodic block, or else we enter a loop.
If we enter a loop, that is if there exist $m_1 < m_2$ and $i_1, i_2, j_1, j_2$ such that $u_{m_1,i_1,j_1} = u_{m_2,i_2,j_2}$, then $u$ is a prime periodic block.
To see this, note that $u_{m_2,i_2,j_2}$ is legal, and there exists some $i$ and $k$ for which $\alpha^i((u_{m_1,i_1,j_1})^k) \in \vartheta(u_{m_2,i_2,j_2})$, but as $u_{m_1,i_1,j_1} = u_{m_2,i_2,j_2}$, then we can iterate this and so there exists an increasing sequence of integers $k_0, k_1, \ldots$ such that some cyclic permutation of the words $(u_{m_1,i_1,j_1})^{k_0}, (u_{m_1,i_1,j_1})^{k_1}, \ldots$ are all legal.
It follows that  $(u_{m_1,i_1,j_1})^\infty$ is in $X_\vartheta$ and, as $\alpha^i((u_0)^j) \in \vartheta^{m_1}(u_{m_1,i_1,j_1})$ for some $i, j$, we also have $(\alpha^i(u_0)^j)^\infty \in X_\vartheta$, hence $(u_0)^\infty \in X_\vartheta$ and so $u_0 = u$ is a prime periodic block for an element in $X_\vartheta$.
\end{proof}
We illustrate the procedure described in the above proof with an example.
\begin{example}
Let $\vartheta$ be the random period doubling substitution given by
\[
\vartheta \colon a \mapsto \{ab, ba\}, b \mapsto \{aa\}
\]
which is compatible and has disjoint images, so the procedure outlined in the proof of Theorem \ref{THM:enumeration-alg} can be applied.

Take as a test word $u = aabaababa$ which is certainly legal as $u \blacktriangleleft \vartheta^4(b)$.
As $|\vartheta| = 2$, we only need to check if $u^2$ is a concatenation of exact inflation words (and we do not actually need to check if $\alpha(u^2)$ is a concatenation of exact inflation words because $u^2$ is if and only if $\alpha(u^2)$ is).
The word $u^2 = aabaababaaabaababa$ is in $\vartheta(baaababaa)$.
We note that $\alpha^2(baaababaa) = aababaaba = u$.
It follows that we have entered a loop and so $u^\infty \in \Per_9(X_\vartheta)$.

Consider instead the word $v = aaabababa$ which is also legal as $v \blacktriangleleft \vartheta^4(a)$.
We have that $u^2 = aaabababaaaabababa$ is in $\vartheta(baaabbaaa)$.
The word $v_1 = baaabbaaa$ is legal, however $(v_1)^2 = baaabbaaabaaabbaaa$ is not legal because two appearances of $bb$ exist which are separated by an odd distance, hence one of them would have to be an exact inflation word, however $bb \notin \vartheta(\mc A)$.
As $v_1$ (and its cyclic permutes) is the only preimage of any cyclic permutation of $v^2$, it follows that $v^2$ cannot be a valid periodic block for a periodic element in $\Per(X_\vartheta)$.
\end{example}

Obviously, the procedure detailed above is not always the fastest way of showing that a particular word is a periodic block or not.
In practice, there are large classes of words for which other methods are faster.
For example, suppose $\vartheta$ also has constant length $\ell$ and that one already has a list of periodic blocks of length $p$.
Then one can find all periodic words of length $\ell p$ just by substituting all the periodic blocks of length $p$ in all possible ways and including all cyclic permutations of those words.
It is clear that all of these words are themselves periodic blocks.
In order to see why this includes every periodic block of length $\ell p$, note that a valid periodic block of length $\ell p$ must have at least one cyclic permutation which is a concatenation of exactly $p$ inflation words and moreover, by the disjoint image condition, at least one of these cyclic permutations can be written as a concatenation of exactly $p$ inflation words in such a way that the preimage is also a valid periodic block of length $p$ (because all preimages of $u^\infty$ are periodic by disjoint images and there exists at least one legal preimage).

Hence, in practice, when implementing a search for periodic words of a substitution of constant length $\ell$ such as the random period doubling substitution, one tends to treat periods of length a multiple of $\ell$ separately to those without $\ell$ as a divisor.

\begin{example}
Let $\vartheta$ be the random period doubling substitution given by
\[
\vartheta \colon a \mapsto \{ab, ba\}, b \mapsto \{aa\}.
\]
Suppose we have already been given the periodic word $u = aab$ of length $3$.
It can quickly be verified by just considering letter frequencies (as described by Proposition \ref{PROP:periodic-block-letter-counts}) that $u$ and its cyclic permutations are the only periodic blocks of length $3$ for $\vartheta$.

There are four substitutive images of $u$ under $\vartheta$ and they are given by
\[
\vartheta(aab) = \{ababaa, abbaaa, baabaa, babaaa\}
\]
We can then include all cyclic permutations of these words to find that the set of periodic blocks of length $p=6$ for $\vartheta$ is given by
\[
\left \{
\begin{array}{l}
ababaa, babaaa, abaaab, baaaba, aaabab, aababa, abbaaa, \\
bbaaaa, baaaab, aaaabb, aaabba, aabbaa, baabaa, aabaab, abaaba
\end{array}
\right \}
\]
and so $|\Per_6(X_\vartheta)| = 15$.

Notice that there are two elements in $\vartheta(aab)$ which are cyclic permutations of one another, and so it is not true that we can count the number of $6$ periodic words by just enumerating the number of choices we have for applying $\vartheta$ to each letter of $aab$ and then accounting for the number of cyclic permutations of these images, as this will lead to over-counting in general (even after taking into account that the word $baabaa$ is $3$-periodic and so only has three distinct cyclic permutations, not six).

Nevertheless, this method of counting at least gives us an upper bound for the number of periodic blocks of length $2p$ in terms of the number of periodic blocks of length $p$, with the only discrepancy being due to this double counting possibility.
That is,
\[
|\Per_{2p}| \leq  (2^{\frac{2p}{3}+1}p -1)|\Per_p|
\]
where $\Per_p := \Per_p(X_\vartheta)$ 
\end{example}

Using these methods, we can compute $\Per_p$ for any $p$ and from these we can also compute $\Orb_p := \{\text{orbits in $X_\vartheta$ of length $p$}\}$.
For the random period doubling substitution, we compute the following values.

\begin{center}
\begin{tabular}{|l|c|c||l|c|c|}
\hline
$p$ & $|\Per_p|$ & $ |\Orb_p|$ & $p$ & $|\Per_p|$ & $ |\Orb_p|$\\
\hline
3 & 3  & 1 & 24 & 176,391  & 7,334\\
6 & 15  & 2 & 27 & 1,533  & 56\\
9 & 21  & 2 & 30 & 216,030  & 7,179\\
12 & 375  & 30 & 33 & 10,992 & 333\\
15 & 108  & 7 & 36 & 19,375,935 & 538,143\\
18 & 2,427  & 133 & 39 & 24,612 & 631\\
21 & 402  & 19 & 42 & 13,106,514 & 312,050\\
\hline
\end{tabular}
\end{center}


\section*{Acknowledgements}
The author wishes to thank Philipp Gohlke and Timo Spindeler for helpful discussions and Scott Balchin and Franz G\"{a}hler for supporting the project with their programming skills.
This work is supported by the German Research Foundation (DFG) via the Collaborative Research Centre (CRC 1283).

\bibliographystyle{jis}
\bibliography{tilings}

\end{document}